\documentclass[12pt]{amsart}
\usepackage{amsthm, amssymb, color}
\usepackage{amsmath}
\usepackage{amscd}
\usepackage{eucal}
\usepackage{epsfig,multicol}
\usepackage[dvipdfm]{pict2e}
\usepackage{color}

\theoremstyle{plain} \numberwithin{equation}{section}
\newtheorem{theo}{Theorem}[section]
\newtheorem{coro}[theo]{Corollary}
\newtheorem{prop}[theo]{Proposition}
\newtheorem{lemm}[theo]{Lemma}

\theoremstyle{definition}
\newtheorem*{defi}{Definition}
\newtheorem{rema}[theo]{Remark}

\def\C{\mathbb C}

\def\QC{\mathbb C}
\def\fln{{\mathcal Flags}(\C^n)}
\def\Tn{T}
\def\Y{\mathcal{Y}} 
\def\HT{H_{\Tn}^{*}} 
\def\HS{H_{\S1}^{*}} 
\def\S1{S}

\def\q{r}

\def\f{\theta}

\def\xi{p}
\def\bxi{\check{p}}
\newcommand{\hsm}{{\hspace{1mm}}}

\newcommand{\Flags}{{\mathcal{F}\ell ags}}

\DeclareMathOperator{\diag}{diag}

\begin{document}
\title{The equivariant cohomology rings of Peterson varieties}
\author[Y. Fukukawa]{Yukiko Fukukawa} \author[M. Harada]{Megumi Harada} \author[M. Masuda]{Mikiya Masuda}
\address{Department of Mathematics, Osaka City University, Sumiyoshi-ku, Osaka 558-8585, Japan}
\email{yukiko.fukukawa@gmail.com}
\address{Department of Mathematics and Statistics, McMaster University, 1280 Main Street West, Hamilton, Ontario L8S4K1, Canada}
\email{Megumi.Harada@math.mcmaster.ca} 
\address{Department of Mathematics, Osaka City University, Sumiyoshi-ku, Osaka 558-8585, Japan}
\email{masuda@sci.osaka-cu.ac.jp}
\date{\today}
\thanks{The first author was supported by JSPS Research Fellowships for Young Scientists.}
\thanks{The second author was partially supported by an NSERC
  Discovery Grant and an Early Researcher Award from the Ministry of
  Research and Innovation of Ontario.}
\thanks{The third author was partially supported by Grant-in-Aid for Scientific Research 25400095}
\begin{abstract}
The main result of this note gives an efficient presentation of the
$S^1$-equivariant cohomology ring of Peterson varieties (in type $A$)
as a quotient of a polynomial ring by an ideal $\mathcal{J}$, in the spirit 
of the well-known Borel presentation of the cohomology of the flag
variety. Our result simplifies previous presentations given by
Harada-Tymoczko and Bayegan-Harada. In particular, our result gives an
affirmative answer to 
a conjecture of Bayegan and Harada that the
defining ideal $\mathcal{J}$ is generated by quadratics. 
\end{abstract}
\maketitle 

\setcounter{tocdepth}{1}
\tableofcontents

\section{Introduction}

The main result of this paper is an explicit and efficient
presentation of the $S^1$-equivariant cohomology ring\footnote{In this
  note, all cohomology rings are with coefficients in $\C$.} of type $A$ Peterson
varieties in terms of generators and relations, in the spirit of the
well-known Borel presentation of the cohomology of the flag
variety. Our presentation is significantly simpler than the
computations given in \cite{HT2} (respectively \cite{BH}) which uses the Monk formula (respectively Giambelli formula)
for type $A$ Peterson varieties.
In particular, our result gives an affirmative answer 
to the conjecture formulated in \cite[Remark 3.12]{BH} by showing
that the defining ideal
for the $S^1$-equivariant cohomology ring of type $A$ Peterson varieties can be
generated by quadratic polynomials. 

We briefly recall the setting of our results. \textbf{Peterson
 varieties} in type $A$ can be 
defined as the following subvariety $\Y$ of
$\mathcal{F}\ell ags(\C^n)$:
\begin{equation}\label{eq:def intro}
\Y := \{ V_\bullet \hsm \mid \hsm  NV_i \subseteq
V_{i+1} \textup{ for all } i = 1, \ldots, n-1\} 
\end{equation}
where $V_\bullet$ denotes a nested sequence 
$0 \subseteq V_1 \subseteq V_2 \subseteq \cdots \subseteq V_{n-1} \subseteq V_n = \C^n$ of subspaces of $\C^n$ and $\dim_\C V_i = i$ for all $i$ and 
$N: \C^n \to \C^n$ denotes the principal nilpotent operator.
These varieties have been much studied due to its relation to 
the quantum cohomology of the flag variety \cite{Kos96,
  Rie03}. Thus it is natural to study their topology, e.g. the
structure of their (equivariant) cohomology rings. 

There is a natural circle subgroup of $U(n,\C)$ which acts
on $\Y$ (recalled in Section~\ref{sec2}).
The inclusion of $\Y$
into $\Flags(\C^n)$ induces a natural ring homomorphism
\begin{equation}\label{eq:intro-proj}
H^*_T(\mathcal{F}\ell ags(\C^n)) \to H^*_{S^1}(\Y)
\end{equation} 
where $T$ is the subgroup of diagonal matrices of $U(n,\C)$ acting in the
usual way on $\Flags(\C^n)$. The content of this manuscript is to give
an efficient presentation of the equivariant cohomology ring
$H^*_{S^1}(\Y)$. Our proof uses Hilbert series and regular sequences, in a similar spirit to previous work of Fukukawa, Ishida, and Masuda \cite{FukukawaIshidaMasuda-typeA, Fukukawa-G2} which computes the graph cohomology of the GKM graphs of the flag varieties of classical type and of $G_2$. 

This paper is organized as follows. We briefly recall the necessary
background in Section~\ref{sec2}. The main theorem, Theorem~\ref{theo:3.1}, is formulated in
Section~\ref{sec3}. Hilbert series and regular sequences are introduced in Section~\ref{sec4} to prove the main result. 
The proof of one key lemma used in the proof of the
main theorem occupies Section~\ref{sec5}.



\bigskip
\noindent
{\bf Acknowledgments.} 
We thank Satoshi Murai for helpful comments on regular sequences, which greatly improved our paper.

\section{Peterson varieties and $S^1$-fixed points}\label{sec2}

In this section we briefly recall the definitions of our main objects
of study. We also record some key facts. For details and proofs, we refer the reader to \cite{HT2}. 
Since we work exclusively in
Lie type $A$, we henceforth omit it from our terminology.

Let $n$ be a fixed positive integer which we assume throughout is
$\geq 2$. The flag variety $\fln$ is the space of nested subspaces in $\C^n$,  i.e.,
\[
\fln=\{ V_\bullet =( V_0 \subset V_1\subset V_2 \subset \cdots \subset V_n=\C^n) \mid \dim_{\C} V_i = i \}.
\]
Let $N$ be the $n \times n$ principal nilpotent operator which, 
 written with respect to the standard basis of $\C^n$, is associated
 to the matrix with a single $n \times n$ Jordan block of eigenvalue
 $0$: 
\[
       N:=\left(
          \begin{array}{ccccc}
            0 & 1& 0 &\cdots & 0 \\
            0&0&1& \cdots &0  \\
            \vdots & \vdots& \vdots &\ddots & \vdots \\
            0& 0& 0&\cdots  & 1\\
            0 & 0& 0&\cdots & 0 \\
         \end{array}
       \right).
\]
Then the \textbf{Peterson variety} $\Y$ is the subvariety of $\fln$ defined as  
\begin{equation}\label{eq:peterson}
\Y:=\{ V_\bullet \in \fln \mid NV_i \subset V_{i+1} \text{ for all } i=1, 2, \cdots , n-1 \}. 
\end{equation}
The Peterson variety is a (singular) projective variety of complex dimension $n-1$.  

The $n$-dimensional compact torus $\Tn$ consisting of diagonal
$n\times n$ unitary matrices acts on $\fln$ in a natural way.  This
torus action does not preserve $\Y$, but the following circle subgroup of $\Tn$ preserves $\Y$: 
\begin{equation} \label{eq:2.1}
\S1:= \left\{   \diag(g^n,g^{n-1},\dots,g)
\mid g\in\C,\| g\|=1 
      \right\}.
\end{equation}
The $\S1$-fixed points in $\Y$ are the $\Tn$-fixed points in $\fln$ that lie in $\Y$, i.e., 
\begin{equation}\label{eq:2.2}
\Y^{\S1} = \fln^T \cap \Y.
\end{equation}
As is standard, we identify $\fln^{\Tn}$ with the set of permutations
on $n$ letters $S_n$. More specifically, it is straightforward to see
that $V_\bullet$ is in $\fln^T$ precisely when there exists $w \in
S_n$ such that 
\begin{equation} \label{eq:2.3}
V_i= \langle e_{w(1)}, e_{w(2)}, \cdots ,e_{w(i)}\rangle \text{ for $i=1,2,\dots,n$} 
\end{equation}
where $e_k$ denotes the $k$-th element in the standard basis of $\C^n$. 
It is shown in \cite{HT2} that a permutation $w
\in S_n \cong \fln^{\Tn}$ is in $\Y^{\S1}$ precisely when the one-line notation of
$w^{-1}$ (equivalently $w$, since in this case $w=w^{-1}$) is of the form
\begin{equation} \label{eq:2.4}
w=\underbrace{j_1\ j_1-1\cdots 1}_\text{$j_1$ entries}\ \underbrace{j_2\ j_2-1\cdots j_1+1}_\text{$j_2-j_1$ entries} \cdots \underbrace{n\ n-1\cdots j_m+1}_\text{$n-j_m$ entries}, 
\end{equation}
where $1\le j_1 < j_2 < \cdots < j_m < n$ is any sequence of strictly increasing integers.

\section{A ring presentation of $H^*_S(\Y)$} \label{sec3}
 
In this section we formulate our main result, Theorem~\ref{theo:3.1}
below, which gives a ring presentation via generators and relations of
the $S$-equivariant cohomology ring of $\Y$. Our presentation is more
efficient than previous computations of this ring
(cf. Remark~\ref{remark:efficient} below). 

Consider the commutative diagram
\begin{equation} \label{eq:3.1}
\begin{CD}
\HT(\fln) @>\iota_1 >> \bigoplus_{w\in \fln^T=S_n}\HT(w) \\
@V\pi_1 VV @V\pi_2 VV \\ 
\HS(\Y) @>\iota_2 >> \bigoplus_{w \in \Y^{\S1}\subset S_n}\HS(w) \\
\end{CD}
\end{equation}
where the maps are induced from the 
inclusions $\Y\hookrightarrow \fln$, $\Y^{\S1}\hookrightarrow
\fln^{T}$ and $\S1\hookrightarrow T$.  Since $H^{odd}(\fln)$ and
$H^{odd}(\Y)$ vanish, the maps $\iota_1$ and $\iota_2$ above are both
injective.  Moreover, it is known that the map $\pi_1$ above is surjective \cite[Theorem 4.12]{HT2}.  Therefore, $\HS(\Y)$ is isomorphic to the image of $\HT(\fln)$ by $\pi_2\circ \iota_1$. 

Through the map $\diag(g^n,g^{n-1},\dots,g)\to g$, we may identify $\S1$ with the unit circle $S^1$ of $\C$ so that we have an identification 
\[
\HS(\mathrm{pt})=H^*(BS) \cong H^*(BS^1) \cong \QC[t].
\]
The torus $\Tn \subset U(n)$ of diagonal unitary matrices has a natural product decomposition $\Tn \cong (S^1)^n$.  This decomposition identifies $BT$ with $(BS^1)^n$ and induces an identification 
\[
H^*_{\Tn}(\mathrm{pt})=H^*(B\Tn) \cong \bigotimes_{i=1}^n H^*(BS^1)
\cong \QC[t_1,\dots,t_n],
\]
where $t_i$ $(i=1,2,\dots,n)$ denotes the element corresponding to the
fixed generator $t$ of $H^2(BS^1)$.  Then from the explicit
description of $S$ as the subgroup $\diag(g^n, g^{n-1},\dots,g)$ of
$\Tn$ it readily follows that  
\begin{equation} \label{eq:3.2}
\pi_2(t_i)=(n+1-i)t.
\end{equation}

We now briefly recall a well-known ring presentation of the
equivariant cohomology ring $H^*_\Tn(\fln)$. There is a tautological
filtration of the trivial rank $n$ vector bundle over $\fln$
\[
\fln \times \{0\} =U_0 \subset U_1 \subset U_2\subset \cdots \subset U_{n-1} \subset U_n=\fln\times \C^n
\]
where the fiber of $U_i$ over a point in $\fln$ corresponding to a
flag $V_\bullet$ is precisely the vector space $V_i$ of $V_\bullet$.
The bundles $U_i$ are all $\Tn$-equivariant and the quotient bundles
$U_i/U_{i-1}$ $(i=1,2,\dots,n)$ are $\Tn$-equivariant line bundles
over $\fln$.  We define 
\begin{equation} \label{eq:3.3}
\tau_i:=c_1^{\Tn}(U_i/U_{i-1}) \in H^2_{\Tn}(\fln) \quad\text{for $i=1,2,\dots,n$},
\end{equation}
where $c_1^{\Tn}$ denotes the equivariant first Chern class.  Then it
is known (see e.g. \cite[Equation (2.1)]{GHZ-GKM} or \cite{FukukawaIshidaMasuda-typeA}) that 
\begin{equation} \label{eq:3.4}
\HT(\fln) =\QC[\tau_1, \cdots ,\tau_n , t_1 , \cdots ,t_n]/I
\end{equation}
where $I$ is the ideal generated by 
\[
e_i(\tau)-e_i(t) \quad\text{for $i=1,2, \cdots ,n$} 
\]
and $e_i(\tau)$ (resp. $e_i(t)$) denotes the $i$-th elementary
symmetric polynomial in the variables $\tau_1 ,\cdots , \tau_n$
(resp. $t_1 ,\cdots ,t_n$). By slight abuse of notation, in the
discussion below we denote
by $\tau_i, t_i$ the corresponding cohomology classes in $\HT(\fln)$
(i.e. the equivalence classes of $\tau_i, t_i$ in the quotient ring
$\QC[\tau_1,\cdots,\tau_n, t_1, \cdots,t_n]/I$).

It follows from \eqref{eq:2.3} and \eqref{eq:3.3} that for $w\in S_n$
we have
\begin{equation} \label{eq:3.5}
\iota_1(\tau_i)|_w=t_{w(i)}
\end{equation}
and clearly
\begin{equation} \label{eq:3.6}
\iota_1(t_i)|_w=t_i,
\end{equation}
where $*|_w$ denotes the $w$-th component of $*$ in the direct sum
$\bigoplus_{w \in S_n} H^*_T(w)$.

We define  
\begin{equation} \label{eq:3.7}
\xi_k :=\pi_1(\sum_{i=1}^{k}(t_i-\tau_i)) \quad\text{for $k=1,\dots,n$}.
\end{equation}
The following lemma computes the images of the $\xi_k$ in $H^*_S(\Y^S)
\cong \bigoplus_{w \in \Y^S} H^*_S(w)$ under the map $\iota_2$
in~\eqref{eq:3.1}. 

\begin{lemm}\label{lemma:xik}
Let $\xi_k \in H^*_S(\Y)$ for $1 \le k \le n$ be defined as above. Then 
\[
\iota_2(\xi_k)|_{w} =\sum_{i=1}^k(w(i)-i)t. 
\]
\end{lemm}

\begin{proof}
For $w \in \Y^{\S1}$ and $1 \le k \le n$ we have 
\begin{equation} \label{eq:3.8}
\begin{split}
\iota_2(\xi_k)|_{w}&=\iota_2(\pi_1(\sum_{i=1}^k(t_i-\tau_i)))|_w
\textup{ by definition of $\xi_k$} \\
&=\pi_2(\iota_1(\sum_{i=1}^k(t_i-\tau_i)))|_w \quad \textup{ by commutativity
  of~\eqref{eq:3.1}} \\
&=\pi_2(\sum_{i=1}^k(t_i-t_{w(i)})) \quad \textup{ by~\eqref{eq:3.5} and~\eqref{eq:3.6}}\\
&=\sum_{i=1}^k(w(i)-i)t \quad \textup{ by~\eqref{eq:3.2}} 
\end{split}
\end{equation}
as desired. 
\end{proof}

Since $\sum_{i=1}^ni=\sum_{i=1}^nw(i)$,
Lemma~\ref{lemma:xik} immediately implies that $\iota_2(\xi_n)|_w=0$
for any $w\in \Y^{\S1}$. In particular we may conclude 
\begin{equation} \label{eq:3.9}
\xi_n=0 
\end{equation}
since $\iota_2$ is injective.  

The next lemma derives some relations which are satisfied among the
elements $\xi_k \in H^*_S(\Y)$. By slight abuse of notation we denote also by $t$ the element in
$H^*_S(\Y)$ which is the image of $t \in H^*_S(\mathrm{pt})$ under the
canonical map $H^*_S(\mathrm{pt}) \to H^*_S(\Y)$. 
Note that $\iota_2(t) \vert_w = t$ for all $w \in \Y^S$. 

\begin{lemm}\label{lemm:3.1}
Let $\xi_k \in H^*_S(\Y)$ for $1 \le k \le n$ be defined as above. Then 
$\xi_k(\xi_k-\frac{1}{2}\xi_{k-1}-\frac{1}{2}\xi_{k+1}-t)=0$ for $k=1,2,\dots,n-1$. 
\end{lemm}

\begin{proof}
Let $w\in \Y^{\S1}\subset S_n$.  It follows from Lemma~\ref{lemma:xik} that  
\begin{equation} \label{eq:3.10}
\begin{split}
&\iota_2(\xi_k-\frac{1}{2}\xi_{k-1}-\frac{1}{2}\xi_{k+1}-t)|_w\\
=& \sum_{i=1}^k(w(i)-i)t-\frac{1}{2}\sum_{i=1}^{k-1}(w(i)-i)t-\frac{1}{2}\sum_{i=1}^{k+1}(w(i)-i)t-t\\
=& \frac{1}{2}(w(k)-w(k+1)-1)t.
\end{split}
\end{equation}
Since $w$ is in $\Y^S$, we know it must be of the form given
in~\eqref{eq:2.4}. If $k=j_q$ for some $1 \le q \le m$, then
$\sum_{i=1}^k i=\sum_{i=1}^k w(i)$. Otherwise, $w(k+1)=w(k)-1$.
Therefore, for any $w \in \Y^{\S1}$ and for any $k$, either \eqref{eq:3.8} or \eqref{eq:3.10} vanishes.  This implies the lemma because $\iota_2$ is injective. 
\end{proof}

Our main result states that the relations given in
Lemma~\ref{lemm:3.1} are enough to determine the ring structure. 

\begin{theo} \label{theo:3.1}
Let $n$ be a positive integer, $n \ge 2$. Let $\Y \subseteq \fln$ be
the Peterson variety defined in~\eqref{eq:peterson}. Let the circle
group $\S1$ act on $\Y$ as described in Section~\ref{sec2}. Then the
$\S1$-equivariant cohomology ring of $\Y$ can be presented by
generators and relations as follows: 
\[
\HS(\Y) \cong \QC[t, \xi_1,\dots,\xi_{n-1}]/J,
\]
where $J$ is the ideal generated by the quadratic polynomials
\begin{equation}\label{eq:defining relations}
\xi_k(\xi_k-\frac{1}{2}\xi_{k-1}-\frac{1}{2}\xi_{k+1}-t)\quad\text{for $k=1,2,\dots,n-1$}
\end{equation}
where we take $\xi_0=\xi_n=0$. 
\end{theo}

Since $H^{odd}(\Y)=0$ and $\HS(\Y)=H^*(B\S1)\otimes H^*(\Y)$ as
$H^*(B\S1)$-modules, we also obtain the following corollary. 

\begin{coro} \label{coro:3.1}
Let $\bxi_k$ be the restriction of $\xi_k$ to $H^*(\Y)$.  Then 
\[
H^*(\Y)=\QC[\bxi_1,\dots,\bxi_{n-1}]/\check{J},
\]
where $\check{J}$ is the ideal generated by 
\[
\bxi_k(\bxi_k-\frac{1}{2}\bxi_{k-1}-\frac{1}{2}\bxi_{k+1})\quad\text{for $k=1,2,\dots,n-1$}
\]
with $\bxi_0=\bxi_n=0$. 
\end{coro}

\begin{rema}\label{remark:efficient}
In \cite{BH} it is also shown that the equivariant cohomology ring $\HS(\Y)$ is a quotient of a polynomial ring $\C[t, \xi_1, \ldots, \xi_{n-1}]$ (using slightly different notation) by an 
ideal generated by polynomials denoted as $q_{i,\mathcal{A}}$ \cite[Theorem 3.8]{BH}. The ring presentation in \cite{BH} is a simplification of the presentation given in \cite[Theorem 6.12 and Corollary 6.14]{HT2} by decreasing both the number of variables in the polynomial ring and the number of generators of the ideal of relations. In fact, it was conjectured in \cite[Remark 3.12]{BH} that things could be made even simpler, namely, that the ideal of relations for the presentation in \cite[Theorem 3.8]{BH} is in fact generated by just the quadratics. Our Theorem~\ref{theo:3.1} proves that this is in fact the case. Indeed, it is straightforward to see from the definitions in \cite{BH, HT2} that the polynomials~\eqref{eq:defining relations} in Theorem~\ref{theo:3.1} correspond to the $q_{i,\mathcal{A}}$ for the special case $\mathcal{A} = \{i\}$. These are precisely the quadratic polynomials among all the $q_{i,\mathcal{A}}$. 
\end{rema}

 \section{Hilbert series and regular sequences}\label{sec4}

In this section we prove our main result, Theorem~\ref{theo:3.1}, modulo one key lemma whose proof we postpone to Section~\ref{sec5}. 

Since the map $\pi_1$ in the diagram~\eqref{eq:3.1} is known to be surjective, it follows from \eqref{eq:3.4}, \eqref{eq:3.7}, \eqref{eq:3.9} and Lemma~\ref{lemm:3.1} that the natural homomorphism of graded rings  
\begin{equation} \label{eq:4.1}
\varphi\colon \QC[t,\xi_1,\dots,\xi_{n-1}]/J\twoheadrightarrow \HS(\Y)
\end{equation} 
is surjective.  Forgetting the $S$-action, this induces a surjective homomorphism of graded rings 
\begin{equation} \label{eq:4.2}
\check{\varphi}\colon \QC[\bxi_1,\dots,\bxi_{n-1}]/\check{J}\twoheadrightarrow H^*(\Y).
\end{equation} 

We next recall the definition of Hilbert series. Suppose 
$A^*=\bigoplus_{i=0}^\infty A^i$ is a graded module over $\QC$. Then 
its associated Hilbert series $F(A^*,s)$ is defined to be the
formal power series
\[
F(A^*,s):=\sum_{i=0}^\infty (\dim_\QC A^{i})s^i.
\]
When comparing Hilbert series of different rings, we use the notation $\sum a_i s^i \geq \sum b_i s^i$ to mean that $a_i  \geq b_i$ for all $i$. 

In our setting, taking the Hilbert series of both rings appearing in \eqref{eq:4.1} and \eqref{eq:4.2} yields 
\begin{eqnarray} 
F(\QC[t,\xi_1,\dots,\xi_{n-1}]/J,s)&\ge F(\HS(\Y),s) \label{eq:4.3}\\
F(\QC[\bxi_1,\dots,\bxi_{n-1}]/\check{J},s)&\ge F(H^*(\Y),s) \label{eq:4.4}
\end{eqnarray}
since both $\varphi$ and $\check{\varphi}$ are surjective. 
Note that $\varphi$ (resp. $\check{\varphi}$) is an isomorphism if and only if the inequality in~\eqref{eq:4.3} (resp. ~\eqref{eq:4.4}) is in fact an equality. 

The Hilbert series of the right hand sides of~\eqref{eq:4.3} and~\eqref{eq:4.4} are known to be as follows. 
It is shown in \cite{ST} that 
\begin{equation} \label{eq:4.5}
F(H^*(\Y),s)=(1+s^2)^{n-1}.
\end{equation}
Moreover, since $\HS(\Y)=H^*(B\S1)\otimes H^*(\Y)$ as $H^*(B\S1)$-modules, \eqref{eq:4.5} implies 
\begin{equation} \label{eq:4.5-1}
F(\HS(\Y),s)=\frac{(1+s^2)^{n-1}}{1-s^2}.
\end{equation}  

The following lemma computes the left hand side of~\eqref{eq:4.4}. Its proof will be given in Section~\ref{sec5} in a more general setting. 

\begin{lemm} \label{lemm:4.1}
$F(\QC[\bxi_1,\dots,\bxi_{n-1}]/\check{J},s)=(1+s^2)^{n-1}$.  
\end{lemm}

Assuming Lemma~\ref{lemm:4.1}, we now complete the proof of Theorem~\ref{theo:3.1}. For this we use the following notion from commutative algebra (see e.g. \cite{stan96}). 

\begin{defi}\label{def:regular}
Let $R$ be a graded commutative algebra over $\QC$ and let $R_+$
denote the positive-degree elements in $R$.  Then a homogeneous
sequence $\f_1,\dots,\f_r \in R_+$ is a \emph{regular sequence} if 
$\f_k$ is a non-zero-divisor in the quotient ring $R/(\f_1,\dots,\f_{k-1})$ for every $1\le k\le r$.  This is equivalent to saying that $\f_1,\dots,\f_r$ is algebraically independent over $\QC$ and $R$ is a free $\QC[\f_1,\dots,\f_r]$-module.  
\end{defi}
It is a well-known fact (see for instance \cite[p.35]{stan96}) that 
a homogeneous sequence $\f_1,\dots,\f_r \in R_+$ 
is a regular sequence if and only if 
\begin{equation} \label{eq:4.6}
F(R/(\f_1,\dots,\f_r),s)=F(R,s)\prod_{k=1}^r(1-s^{\deg{\f_k}}). 
\end{equation}
A sketch of the proof of this fact is as follows.  Let $\f_1,\dots,\f_r$ be a homogeneous sequence of $R$ and set $R_k:=R/(\f_1,\dots,\f_k)$ for $1\le k\le r$.  Consider the exact sequence 
\begin{equation*} \label{eq:4.6-a}
R_{k-1}\stackrel{\times \f_k}\longrightarrow R_{k-1}\to R_k\to 0 \quad \text{for $1\le k\le r$},
\end{equation*}
where $\times \f_k$ denotes multiplication by $\f_k$,  the map $R_{k-1}\to R_k$ is the quotient map and $R_0:=R$.  The regularity of the sequence $\f_1,\dots,\f_r$ implies that the map $\times\f_k$ is injective for every $1\le k\le r$, which in turn implies 
\begin{equation*} \label{eq:4.6-b}
F(R_k,s)=F(R_{k-1},s)(1-s^{\deg\f_k}) \quad\text{for any $1\le k\le r$}.
\end{equation*} 
The desired fact then follows.

Returning to our setting, we have the following lemma. 

\begin{lemm} \label{lemm:4.1-1}
In the polynomial ring $\QC[t,\xi_1,\dots,\xi_{n-1}]$, the sequence 
\[
\begin{split}
\f_k:&=\xi_k(\xi_k-\frac{1}{2}\xi_{k-1}-\frac{1}{2}\xi_{k+1}-t) \quad \text{for $1\le k\le n-1$},\\
\f_n:&=t.
\end{split}
\]
is regular. 
\end{lemm}

\begin{proof}  
Since $\f_n=t,$ 
from the definitions of $\f_k$ and the ideals $J$ and $\check{J}$ given in the statements of Theorem~\ref{theo:3.1} and Corollary~\ref{coro:3.1} it follows that 
\[
\begin{split}
&F(\QC[t,\xi_1,\dots,\xi_{n-1}]/(\f_1,\dots,\f_{n-1},\f_n),s)\\
=&F(\QC[\bxi_1,\dots,\bxi_{n-1}]/\check{J},s)\\
=&(1+s^2)^{n-1}
\end{split}
\]
where the last equality follows from Lemma~\ref{lemm:4.1}. This implies that \eqref{eq:4.6} is satisfied in our setting because $\deg \f_i=4$ for $1\le i\le n-1$, $\deg \f_n=2$ and 
\begin{equation} \label{eq:4.6-2}
F(\QC[t,\xi_1,\dots,\xi_{n-1}],s)=\frac{1}{(1-s^2)^n}. 
\end{equation} 
The result follows. 
\end{proof}

We can now prove the main theorem. 

\begin{proof}[Proof of Theorem~\ref{theo:3.1}]
From the definition of a regular sequence it is clear that the subsequence $\f_1,\dots,\f_{n-1}$ of a regular sequence $\f_1,\dots,\f_n$ is again a regular sequence. Hence it follows from \eqref{eq:4.6} and \eqref{eq:4.6-2} that 
\[
\begin{split}
F(\QC[t,\xi_1,\dots,\xi_{n-1}]/J,s)&=F(\QC[t,\xi_1,\dots,\xi_{n-1}]/(\f_1,\dots,\f_{n-1}),s)\\
&=\frac{1}{(1-s^2)^n}\prod_{k=1}^{n-1}(1-s^{\deg \f_k})\\
&=\frac{(1+s^2)^{n-1}}{1-s^2}.
\end{split}
\]
This together with \eqref{eq:4.5-1} shows that the equality holds in \eqref{eq:4.3}. Hence the map $\varphi$ in \eqref{eq:4.1} is an isomorphism, as desired. 
\end{proof}

\section{Proof of Lemma~\ref{lemm:4.1}}\label{sec5}

This section is devoted to the proof of Lemma~\ref{lemm:4.1}.  
Note first that Lemma~\ref{lemm:4.1} is equivalent to the statement that the sequence of homogeneous elements 
\[
\bxi_k(\bxi_k-\frac{1}{2}\bxi_{k-1}-\frac{1}{2}\bxi_{k+1})\quad\text{$(k=1,2,\dots,n-1)$},
\]
(where $\bxi_0=\bxi_n$ are both defined to be $0$) is a regular sequence in the polynomial ring $\QC[\bxi_1,\dots,\bxi_{n-1}]$.  
We now recall a criterion which characterizes when such a homogenous sequence in a polynomial ring is regular. We learned this criterion from S. Murai. 

\begin{prop} \label{prop:5.1} 
A sequence of positive-degree
  homogeneous elements $\f_1,\dots,\f_{\q}$ in the polynomial ring
  $\C[z_1,\dots,z_{\q}]$ is a regular sequence if and only if the
  solution set in $\C^{\q}$ of the equations $\f_1=0,\dots,\f_{\q}=0$
  consists only of the origin $\{0\}$. 
\end{prop}

\begin{proof}
First we claim that the homogeneous sequence $\f_1,\dots,\f_{\q}$ is regular
if and only if the Krull dimension of
$\C[z_1,\dots,z_{\q}]/(\f_1,\dots,\f_{\q})$ is zero. To see this,
observe that by definition, if $\f_1, \dots,
\f_{\q}$ is a regular sequence then the $\f_1, \dots, \f_{\q}$ are
algebraically independent. This implies that the Krull dimension of
$\C[z_1, \dots, z_{\q}]/(\f_1, \dots, \f_{\q})$ is zero (note that
the number of generators of the polynomial ring
$\C[z_1,\dots,z_{\q}]$ is equal to the length of the regular
sequence). In the other direction, if $\C[z_1,\dots, z_{\q}]/(\f_1,
\dots, \f_{\q})$ has Krull dimension $0$, then the $\f_1, \dots,
\f_{\q}$ are a homogeneous system of parameters for
$\C[z_1,\dots,z_{\q}]$ \cite[Definition 5.1]{stan96}. Moreover, since
the polynomial ring 
$\C[z_1,\dots,z_{\q}]$ is Cohen-Macaulay, by \cite[Theorem
5.9]{stan96} we may conclude that the homogeneous system of parameters
$\f_1, \dots, \f_{\q}$ is a regular sequence. 

Next we observe that by Hilbert's Nullstellensatz the quotient ring
\[
\C[z_1,\dots,z_{\q}]/(\f_1,\dots,\f_{\q})
\]
has Krull dimension $0$ if
and only if the algebraic set in $\C^{\q}$ defined by the equations
$\f_1=0,\dots,\f_{\q}=0$ is zero-dimensional. Since the polynomials
$\f_1,\dots,\f_{\q}$ are assumed to be homogeneous, the corresponding
zero-dimensional algebraic set in $\C^{\q}$ must consist of only the
origin. This proves the proposition. 
\end{proof}

By Proposition~\ref{prop:5.1}, in order to prove Lemma~\ref{lemm:4.1} it suffices to check that the solution set in $\C^{\q}$ of the equations 
\begin{equation} \label{eq:5.0}
z_i^2=\frac{1}{2}z_i(z_{i-1}+z_{i+1}) \quad\text{$(i=1,2,\dots,\q)$}
\end{equation}
(where $z_0=z_{\q+1}=0$) consists of only the origin.  To prove this, we consider a more general set of 
equations in $\QC^{\q}$ ($\q\ge 2$), namely:   
\begin{equation} \label{eq:5.1}
\begin{split}
z_1^2&=b_1z_1z_2\\
z_i^2 &=z_i(a_{i-1}z_{i-1} +b_iz_{i+1}) \quad (i = 2,\cdots, \q-1)\\
z_{\q}^2&=a_{\q-1}z_{\q-1}z_{\q}
\end{split}
\end{equation}
where $a_i,b_i$ for $i=1,2,\dots,\q-1$ are fixed complex numbers. 

\begin{lemm} \label{lemm:5.1}
In the setting above, set $c_i :=a_ib_i$ for $i=1,2,\dots,\q-1$.  If 
\begin{equation} \label{eq:5.2}
 1-\cfrac{c_i}{1-\cfrac{c_{i+1}}{\qquad \cfrac{\ddots}{1-\cfrac{c_{j-1}}{1-c_j}}}}\not=0
\end{equation}
for all $1\le i\le j\le \q-1$, 
then the solution set of the equations \eqref{eq:5.1} consists of only the origin in $\QC^{\q}$. 
\end{lemm}

\begin{proof}
We prove the lemma by induction on $\q$, the number of variables.  It is easy to check the lemma directly for the base case $\q=2$.  Now suppose that $\q\ge 3$ and the result of the lemma holds for $\q-1$. 
Note that the equations in \eqref{eq:5.1} which involve the variable $z_{\q}$ are the two equations 
\begin{equation*}
\begin{split}
z_{\q-1}^2 &=z_{\q-1}(a_{\q-2}z_{\q-2}+b_{\q-1}z_{\q}) \\ 
z_{\q}^2 & =a_{\q-1}z_{\q-1}z_{\q}.
\end{split}
\end{equation*}
From the latter equation we can conclude that either $z_{\q}=0$ or $z_{\q}=a_{\q-1}z_{\q-1}$.  

Now we take cases. Suppose $z_{\q}=0$. Then the equations \eqref{eq:5.1} become 
\[
\begin{split}
z_1^2&=b_1z_1z_2\\
z_i^2 &=z_i(a_{i-1}z_{i-1} +b_iz_{i+1}) \quad (i = 2,\cdots, \q-2)\\
z_{\q-1}^2&=a_{\q-2}z_{\q-2}z_{\q-1}. 
\end{split}
\]
By the induction assumption, 
the solution set of these equations consists of only the origin since \eqref{eq:5.2} is satisfied for all $1\le i\le j \le \q-2$. 

Next suppose $z_{\q}=a_{\q-1}z_{\q-1}$. In this case the equations \eqref{eq:5.1} turn into 
\begin{equation} \label{eq:5.3}
\begin{split}
z_1^2&=b_1z_1z_2\\
z_i^2 &=z_i(a_{i-1}z_{i-1} +b_iz_{i+1}) \quad (i = 2,\cdots, \q-2)\\
z_{\q-1}^2&=\frac{a_{\q-2}}{1-a_{\q-1}b_{\q-1}}z_{\q-2} z_{\q-1}. 
\end{split}
\end{equation} 
Here we know that $1-a_{\q-1}b_{\q-1}\not=0$ from the condition \eqref{eq:5.2} with $i=j=\q-1$.  Again by the induction assumption, the solution set of the equations \eqref{eq:5.3} consists of only the origin if 
\begin{equation} \label{eq:5.4}
 1-\cfrac{c'_i}{1-\cfrac{c'_{i+1}}{\qquad \cfrac{\ddots}{1-\cfrac{c'_{j-1}}{1-c'_j}}}}\not= 0 \quad\text{for all $1\le i\le j\le \q-2$},
\end{equation}
where 
\[
c'_k=c_k \  (1\leq k \leq \q-3),\qquad c'_{\q-2}=\frac{c_{\q-2}}{1-c_{\q-1}}.
\]
From the definition of the $c'_k$ it is clear that~\eqref{eq:5.4} is
equivalent to~\eqref{eq:5.2} for $i$ and $j$ with $1 \le i \le j \le
\q-3$. Further, the case $i=j=\q-2$ of~\eqref{eq:5.4} follows from the
$i=\q-2, j=\q-1$ case of~\eqref{eq:5.2}, and the case $i<j=\q-2$
of~\eqref{eq:5.4}
follows from the $i\leq j=\q-1$ case of~\eqref{eq:5.2}. Hence~\eqref{eq:5.4} holds for all choices of $i$ and $j$ and by the induction assumption the solution set consists of only the origin, as desired. 
\end{proof}

\begin{rema}
It is not difficult to see that the \lq\lq only if" part of Lemma~\ref{lemm:5.1} also holds, but we do not need this implication in what follows. 
\end{rema} 

We now return to our special case, for which $a_i=b_i=1/2$ and hence $c_i=1/4$ for all $1 \le i \le \q-1$.  Below, we give a sufficient condition for \eqref{eq:5.2} to be satisfied when $c_i=a_ib_i$ $(i=1,2,\dots,\q-1)$ is a constant $c$ independent of $i$.  This will suffice to prove Lemma~\ref{lemm:4.1}. For this purpose, consider the numerical sequence $\{x_m\}_{m=0}^\infty$ defined by the following recurrence relation and with $x_0=1$: 
\begin{equation} \label{eq:5.5}
x_{m}=1-\frac{c}{x_{m-1}} \quad\text{for $m\ge 1$}.  
\end{equation}
In the situation when the $c_i$ are all equal, it is straightforward to see that the condition \eqref{eq:5.2} is equivalent to the statement that $x_m\not=0$ for $m=1,2,\dots,\q-1$.  
We have the following. 

\begin{lemm}\label{lemm:5.2}
Let $\{x_m\}$ be the sequence defined in~\eqref{eq:5.5}. Then: 
\begin{enumerate}
\item if $0\le c\le 1/4$, then $x_m\ge (1+\sqrt{1-4c})/2$ for any $m\ge 1$, and 
\item if $c< 0$, then $x_m\ge 1$ for any $m\ge 1$.
\end{enumerate}
In particular, if $c$ is any real number $\le 1/4$, then $x_m>0$ for all  $m\ge 1$.  
\end{lemm}

\begin{proof}
Let $0\le c\le 1/4$ and suppose that 
\begin{equation} \label{eq:5.6}
x_{m-1}\ge \frac{1+\sqrt{1-4c}}{2} > 0 \quad\text{for some $m\ge 1$}. 
\end{equation}
Then it follows from \eqref{eq:5.5} and \eqref{eq:5.6} that 
\[
x_{m}=1-\frac{c}{x_{m-1}}\ge 1-\frac{2c}{1+\sqrt{1-4c}}=\frac{1+\sqrt{1-4c}}{2}.
\]
This proves (1) in the lemma since the inequality \eqref{eq:5.6} is satisfied for $m=1$. A similar argument proves (2).  
\end{proof}

The proof of Lemma~\ref{lemm:4.1} is now straightforward. 

\begin{proof}[Proof of Lemma~\ref{lemm:4.1}]
The statement of Lemma~\ref{lemm:4.1} follows from Proposition~\ref{prop:5.1}, Lemma~\ref{lemm:5.1} and Lemma~\ref{lemm:5.2}.
\end{proof}

\end{document}